\documentclass[11pt,a4paper]{article}
\usepackage{amssymb,amsmath,amsopn,amsthm,graphicx,makeidx}
\usepackage[all,2cell]{xy} \UseAllTwocells

\begin{document}

\newtheorem{definition}{Definition}[section]
\newtheorem{definitions}[definition]{Definitions}
\newtheorem{lemma}[definition]{Lemma}
\newtheorem{prop}[definition]{Proposition}
\newtheorem{theorem}[definition]{Theorem}
\newtheorem{cor}[definition]{Corollary}
\newtheorem{cors}[definition]{Corollaries}
\theoremstyle{remark}
\newtheorem{remark}[definition]{Remark}
\theoremstyle{remark}
\newtheorem{remarks}[definition]{Remarks}
\theoremstyle{remark}
\newtheorem{notation}[definition]{Notation}
\theoremstyle{remark}
\newtheorem{example}[definition]{Example}
\theoremstyle{remark}
\newtheorem{examples}[definition]{Examples}
\theoremstyle{remark}
\newtheorem{dgram}[definition]{Diagram}
\theoremstyle{remark}
\newtheorem{fact}[definition]{Fact}
\theoremstyle{remark}
\newtheorem{illust}[definition]{Illustration}
\theoremstyle{remark}
\newtheorem{rmk}[definition]{Remark}
\theoremstyle{definition}
\newtheorem{question}[definition]{Question}
\theoremstyle{definition}
\newtheorem{conj}[definition]{Conjecture}

\renewcommand{\marginpar}[1]{ }

\renewenvironment{proof}{\noindent {\bf{Proof.}}}{\hspace*{3mm}{$\Box$}{\vspace{9pt}}}

\title{Almost dual pairs and definable classes of modules}
\author{Akeel Ramadan Mehdi\footnote{akeel\_math@yahoo.com} and Mike Prest\footnote{mprest@manchester.ac.uk},\\ School of Mathematics \\ Alan Turing Building
\\University of Manchester\\
Manchester M13 9PL\\UK}

\maketitle

\abstract{In \cite{Holm}  Holm considers categories of right modules dual to those with support in a set of finitely presented modules. We extend some of his results by placing them in the context of elementary duality on definable subcategories.}

\section{Introduction}\label{intro}

Let $ {\mathcal B}={\rm add}({\mathcal B}) $ be an additive subcategory of $ R\mbox{-}{\rm mod}, $ the category of finitely presented left $ R$-modules.  Lenzing \cite{LenzTrans} studied properties of those categories of the form $ \varinjlim {\mathcal B} $, where this denotes the closure of ${\mathcal B}$ under direct limits in the category, $R\mbox{-}{\rm Mod}$, of all left $R$-modules.  Given a left module $M$, denote by $M^\ast$ its dual (right $R$-) module ${\rm Hom}_{\mathbb Z}(M, {\mathbb Q}/{\mathbb Z})$.  Holm (\cite{Holm})  considers the closure, ${\rm Prod}({\mathcal B}^\ast)$, of ${\mathcal B}^\ast $ under direct products and direct summands in ${\rm Mod}\mbox{-}R$, which he refers to as the category of modules with cosupport in ${\mathcal B}$ (his notation, which we will not use, for this is $ ({\rm Mod}\mbox{-}R)^{{\mathcal B}} $).  In this paper we set the duality between categories such as $ \varinjlim {\mathcal B} $ and $ {\rm Prod}({\mathcal B}^\ast )$ in a more general context and we extend some of the results from Holm's paper.

Every dual of a module is pure-injective and the classes of pure-injectives which are closed under products and direct summands correspond bijectively (by taking their closures under pure submodules) to the type-definable classes of modules considered by Burke (\cite{BurThes}) and hence also to the closed subsets in the full support topology that he defined on the set of indecomposable pure-injective modules.  These type-definable classes were introduced as an extension of the definable classes which arose in the model theory of modules (\cite{Zie}, \cite{PreBk}).  From this perspective it is natural to extend results from \cite{Holm} by considering the closure of classes under pure submodules.  There is a duality, elementary duality, between definable classes (\cite{HerzDual}) and, again, this provides a perspective which allows us to clarify the relation between classes such as $ \varinjlim {\mathcal B} $ and $ {\rm Prod}({\mathcal B}^\ast )$.

\vspace{4pt}

For a class ${\mathcal X}$ of modules (we always assume our classes to be closed under isomorphism) we write ${\rm add}({\mathcal X})$ (respectively ${\rm Add}({\mathcal X})$) for the closure of ${\mathcal X}$ under finite (resp.~arbitrary) direct sums and direct summands, we set ${\mathcal X}^\ast =\{ M^\ast: M\in {\mathcal X}\}$ to be the class of duals of modules in ${\cal X}$, we denote by ${\rm Prod}({\mathcal X})$ the closure of ${\mathcal X}$ under direct products and direct summands and by ${\rm Prod}^+({\mathcal X})$ we denote the closure of ${\mathcal X}$ under direct products and pure submodules.  We also write ${\mathcal X}^+$ for the closure of ${\mathcal X}$ under pure submodules.  We write ${\rm Pinj}({\cal X})$ for the class of pure-injective modules which are in ${\cal X}$.  By ${\rm pinj}_R$ we denote the set of isomorphism classes of indecomposable pure-injective right $R$-modules.  We will use \cite{PreNBK} as a handy reference for definitions and results around purity; there are many other sources.

Recall, e.g.~\cite[4.3.29]{PreNBK}, than any character/dual module $M^\ast$ is pure-injective.

\vspace{4pt}

Let ${\mathcal S}$, respectively ${\mathcal P}$, denote subclasses (or subcategories) of $R\mbox{-}{\rm Mod}$, respectively ${\rm Mod}\mbox{-}R$.  We say that $({\mathcal S},{\mathcal P})$ is an {\bf almost dual pair} if:

1.  ${\mathcal P}={\rm Prod}({\mathcal P})$ and ${\mathcal P}$ is closed under pure-injective hulls\footnote{In \cite{MehdiThes}, and also in Holm and J\o rgensen's notion of a duality pair \cite{HolmJorg}, the class ${\cal P}$ is not required to be closed under pure-injective hulls.}

2.  $M\in {\mathcal S}$ iff $M^\ast \in {\mathcal P}$.

Immediate examples include:  $(R\mbox{-}{\rm Mod}, {\rm Pinj}({\rm Mod}\mbox{-}R))$;  $(R\mbox{-}{\rm Flat}, {\rm Abs}\mbox{-}R)$, $(R\mbox{-}{\rm Flat}, {\rm Inj}\mbox{-}R)$ (\cite[p.~239]{Lambek}) where we use obvious notation for flat, absolutely pure (=fp-injective) and injective modules; the pair $(R\mbox{-}{\rm Abs},{\rm Flat}\mbox{-}R)$ is almost dual iff $R$ is left coherent (\cite[1.6]{Wurf}).

\begin{lemma}\label{Sclosed}\marginpar{Sclosed} Suppose that $({\mathcal S},{\mathcal P})$ is an almost dual pair.  If $0\rightarrow L\rightarrow M\rightarrow N\rightarrow 0$ is a pure-exact sequence, then $M \in {\mathcal S}$ iff $L,N\in {\cal S}$.  Moreover ${\mathcal S}={\rm Add}({\mathcal S}) = \varinjlim {\mathcal S}$.
\end{lemma}
\begin{proof} If $0\rightarrow L\rightarrow M\rightarrow N\rightarrow 0$ is a pure-exact sequence with $M \in {\mathcal S}$ then, applying ${\rm Hom}_{\mathbb Z}(-,{\mathbb Q}/{\mathbb Z})$, we obtain the split exact sequence (e.g.~\cite[4.3.30]{PreNBK}) $0\rightarrow N^\ast \rightarrow M^\ast  \rightarrow L^\ast \rightarrow 0$, from which we see that $M^\ast \in {\mathcal P}$ iff $L^\ast, N^\ast \in {\cal P}$, whence we obtain the first statement.

Also, since $(\bigoplus_{i\in I}M_i)^\ast = \prod_{i\in I} M_i^\ast$, the first equality of the second statement is immediate from the definitions and since, for any directed system $(M_i)_i$ the canonical map $\bigoplus_i M_i \rightarrow \varinjlim_i M_i$ is a pure epimorphism (\cite[p.~56]{RayGru}), the second equality follows from the first assertion.
\end{proof}

\begin{cor}\label{Sprecov}\marginpar{Sprecov} If  $({\mathcal S},{\mathcal P})$ is an almost dual pair then ${\cal S}$ is (pre-)covering in $R\mbox{-}{\rm Mod}$.
\end{cor}

This follows directly by \cite[2.5]{HolmJorg}, \cite[Thm.~4]{KraAdj}.

Although ${\mathcal S}$ is completely determined by ${\mathcal P}$ the converse is not true:  if $({\mathcal S},{\mathcal P})$ is an almost dual pair then both $({\mathcal S},{\rm Pinj}({\mathcal P}))$ and $({\mathcal S},{\mathcal P}^+)$ are almost dual pairs and these are equal iff ${\cal P}^+$ consists only of pure-injectives, but that is a strong condition (which we consider in Section \ref{secfp}), being equivalent to $\Sigma$-pure-injectivity of every member of ${\cal P}$.  We will see that an additional condition is needed for these to be the upper and lower bounds of the possibilities for ${\cal P}$.

Indeed, if $({\mathcal S},{\mathcal P})$ is an almost dual pair and if ${\rm Prod}({\mathcal S}^\ast ) = {\rm Pinj}({\cal P})$ then we will say that this is a {\bf dual pair}\footnote{Note that this is considerably more restrictive than the similarly-named notion of duality pair from \cite{HolmJorg}.}; in this case it follows directly that ${\rm Pinj}({\mathcal S}^\ast) \subseteq {\mathcal P} \subseteq ({\mathcal S}^\ast)^+$.  All the examples above are actually dual pairs but we will give an example later (\ref{notdual}) of an almost dual pair with ${\rm Prod}({\mathcal S}^\ast )$ properly contained in ${\rm Pinj}({\cal P})$.  In order to establish that example, we need to distinguish between arbitrary classes of pure-injectives closed under products and direct summands and those which arise by closing classes of duals of modules under these operations; we do this in Section \ref{secdual}.  In that section we also show that the definition of almost dual pair is independent of the choice of duality - for example, if $R$ is an algebra over a field $K$ then it would be natural to use ${\rm Hom}_K(-,K)$ in place of ${\rm Hom}_{\mathbb Z}(-, {\mathbb Q}/{\mathbb Z})$, or one might use the ``local" duality $M \mapsto {\rm Hom}_S(M,E)$ where $S={\rm End}(M)$ and $E$ is an injective cogenerator for $S$-modules.

First, however, we note the connection with torsion theories on the associated functor category and with topologies on ${\rm pinj}_R$.

\section{Type-definable subcategories and torsion theories on the functor category}

A class ${\mathcal X}$ of (right $R$-)modules is said to be {\bf type-definable} (\cite{BurThes}, see \cite[\S\S5.3.7, 12.7]{PreNBK}) if it is closed under pure-injective hulls, direct products and pure submodules (the actual definition is in terms of pp-types but this is equivalent).  If a class ${\cal X}$ is type-definable then ${\rm Pinj}({\cal X})$ is a class ${\cal P}$ of pure-injectives satisfying ${\cal P}={\rm Prod}({\cal P})$ and every such class of pure-injectives arises in this way from a type-definable class.  There is a natural bijection between these classes and hereditary torsion theories on the (locally coherent, Grothendieck abelian) functor category $(R\mbox{-}{\rm mod}, {\bf Ab})$.

This bijection is induced by the full embedding $\epsilon :{\rm Mod}\mbox{-}R \rightarrow (R\mbox{-}{\rm mod}, {\bf Ab})$ which takes $M_R$ to the functor $(M\otimes_R-)$ and which has the natural action on morphisms.  The image under $\epsilon$ of an exact sequence is exact iff the original sequence is pure-exact and the image of a module is injective iff the original module is pure-injective (see \cite[\S 1]{GrJeDim}, \cite[B16]{JeLe} or \cite[\S12.1]{PreNBK}).  Hereditary torsion theories on Grothendieck abelian categories are in bijection with classes ${\cal I}$ of injective objects closed under direct products and direct summands - the torsionfree class being the class of subobjects of objects in ${\cal I}$ -  so we have the following.

\begin{remark}\label{tpdef}\marginpar{tpdef} There are natural bijections (as described above) between:

\noindent type-definable classes of right $R$-modules;

\noindent classes ${\cal P}$ of pure-injective right $R$-modules satisfying ${\cal P} ={\rm Prod}({\cal P})$

\noindent hereditary torsion theories on $(R\mbox{-}{\rm mod}, {\bf Ab})$.
\end{remark}

We will see (\ref{notdual}) that not every such class ${\cal P}={\rm Prod}({\cal P})$ of pure-injectives has the form ${\rm Prod}({\cal S}^\ast)$ for some class ${\cal S}$ of modules.

\begin{remark}\label{proddual}\marginpar{proddual} There are natural bijections  between the following:

\noindent classes ${\cal P}$ of pure-injective right $R$-modules of the form ${\rm Prod}({\cal S}^\ast)$ for some class ${\cal S}$ of left $R$-modules;

\noindent dual pairs $({\cal S}, {\cal P})$ with minimal ${\cal P}={\rm Prod}({\cal S}^\ast)$, that is, ${\cal P}={\rm Pinj}({\cal P})$;

\noindent dual pairs $({\cal S}, {\cal P})$ with maximal ${\cal P}={\rm Prod}({\cal S}^\ast)^+$, that is ${\cal P}={\cal P}^+$;
\end{remark}

A subcategory of a module category is a {\bf definable subcategory} if it is closed under direct products, direct limits and pure submodules, equivalently if it is type-definable and closed under direct limits.  In the bijection above these correspond exactly to the torsion theories of finite type (see \cite[12.4.1]{PreNBK}), meaning that the torsion class is generated by finitely presented objects.

The results in this paper could be presented, as Holm does to some extent in \cite{Holm}, in a way which makes use of this functor-category perspective.

\section{Dualities}\label{secdual}\marginpar{secdual}

If $M$ is a left $R$-module, $S\rightarrow {\rm End}(M)$ is a ring homomorphism and $E$ is an injective cogenerator for right $S$-modules then we will say that ${\rm Hom}_S(-,E_S)$ is a duality that applies to $M$ and we will write $M^\ast$ for the right $R$-module ${\rm Hom}_S(M_S,E_S)$.  In this section we will use the notions of pp formula and pp-type and associated results from the model theory of modules (\cite{PreNBK} is one reference for these) because these apply nicely to the relation between $M$ and $M^\ast$.  In particular we need the following, for which see \cite[\S 2(c)]{Z-HZPSS}, \cite[1.5]{PRZ2} (or \cite[1.3.12]{PreNBK}).

\begin{prop}\label{dualann}\marginpar{dualann} Let $M$ be any left $R$-module, let $\phi(v)$ be a pp formula (with one free variable) for right $R$-modules and let $^*$ be a duality that applies to $M$.  Then the solution set, $\phi(M^\ast)$, of $\phi$ in $M^\ast$ is the annihilator $\{ f\in M^\ast: f\cdot D\phi(M)=0\}$ of the solution set of $D\phi$ in $M$.  That is, $f\in \phi(M^\ast)$ iff $D\phi(M)\leq {\rm ker}(f)$.
\end{prop}

Here $D\phi$ is the elementary dual of the formula $\phi$ (\cite{PreDual}, see \cite[\S 1.3]{PreNBK}).  Duality applied twice is equivalent to the identity - $DD\phi(M) =\phi(M)$ for every module $M$ - so the result with the roles of $\phi$ and $D\phi$ reversed also is true.  We freely use the fact that for every module $M$ and pp formula $\phi$, it is the case that $\phi(M)$ is an ${\rm End}(M)$-submodule of $M$ (see \cite[1.1.8]{PreNBK}).

\begin{theorem}\label{anydual}\marginpar{anydual} Suppose that $^\ast$ and $^\sharp$ are dualities each of which applies to the module $M$.  Then ${\rm Prod}(M^\ast) ={\rm Prod}(M^\sharp)$.
\end{theorem}
\begin{proof} It is enough to show that $M^\ast \in {\rm Prod}(M^\sharp)$.  To establish that, it will enough to prove that for each nonzero $f\in M^\ast$ there is $g\in (M^\sharp)^I$ for some set $I$, such that the {\bf pp-type of} $f$ {\bf in} $M^\ast$ - the set of all pp formulas $\phi$ such that $f\in \phi(M^\ast)$ - equals the pp-type of $g$ in $(M^\sharp)^I$.  For then there will, by \cite[4.3.9]{PreNBK}, be a morphism $\alpha_f$ from $M^\ast$ to $(M^\sharp)^I$ taking $f$ to $g$.  The product over all $f\in M^\ast$ of these maps $\alpha_f$ will then be a pure, hence split, embedding of $M^\ast$ into a direct product of copies of $M^\sharp$.  Let us write $p$ for the pp-type of $f$ in $M^\ast$.

By \ref{dualann}, for each $\phi \in p$, $D\phi(M)\leq {\rm ker}(f)$, therefore $\sum_{\phi \in p} D\phi(M)\leq {\rm ker}(f)$, and for each pp formula (in the same free variable $v$) $\psi$ not in $p$ (let us write $p^-$ for the set of these), $D\psi(M)\nleq {\rm ker}(f)$, in particular $D\psi(M)\nleq \sum_{\phi \in p} D\phi(M)$.

Therefore since these are $S$-modules and $E$ is an injective cogenerator for $S$-modules (with the notation introduced above), there is $g_\psi \in M^\sharp$ such that $\sum_{\phi \in p} D\phi(M) \leq {\rm ker}(g_\psi)$ and $D\psi(M) \nleq {\rm ker}g_\psi$.  So, by \ref{dualann}, for all $\phi\in p$, $g_\psi \in \phi(M^\sharp)$ but $g_\psi \notin \psi(M^\sharp)$.

Set $g=(g_\psi)_{\psi \in p^-} \in (M^\sharp)^{p^-}$.  Then, since pp formulas commute with direct products (\cite[1.2.3]{PreNBK}), the pp-type of $g$ is precisely $p$, as required.
\end{proof}

\begin{theorem}\label{dualindec}\marginpar{dualindec} Let $M$ be any nonzero module and let $M^\ast$ be a dual of $M$.  Then $M^\ast$ has an indecomposable direct summand.
\end{theorem}
\begin{proof} Choose $a\in M$, $a\neq 0$.  By Zorn's Lemma there is, in the lattice of pp-definable subgroups of $M$, a lattice ideal ${\cal I}$ such that, for all $\psi(M) \in {\cal I}$, $a\notin \psi(M)$ and ${\cal I}$ is maximal such.  For notational simplicity, let us write $\psi \in {\cal I}$ rather than $\psi(M) \in {\cal I}$.  If $\phi$ is such that $\phi(M)\notin {\cal I}$ then, by maximality of ${\cal I}$, $a\in \psi(M)+\phi(M)$ for some $\psi \in {\cal I}$.

Take $f\in M^\ast$ such that $\sum_{\psi\in {\cal I}}\psi(M)\leq {\rm ker}(f)$ and $f(a)\neq 0$.  By \ref{dualann} (with dual formulas on the other side), $f\in D\psi(M^\ast)$ for all $\psi\in {\cal I}$ but, for all $\phi\notin {\cal I}$, $f\notin D\phi(M^\ast)$:  for, given $\phi\notin {\cal I}$, choose $\psi\in {\cal I}$ such that $a\in \psi(M)+\phi(M)= (\psi+\phi)(M)$ so, by \ref{dualann}, $f\notin D(\psi+\phi)(M^\ast) = (D\psi(M^\ast)\cap D\phi(M^\ast))$ so, since $f\in D\psi(M^\ast)$, $f\notin D\phi(M^\ast)$ as claimed.  Therefore the pp-type, $p$, of $f$ in $M^\ast$ is $\{ D\psi: \psi\in {\cal I}\}$.

We claim that this pp-type, $p$, is {\bf irreducible} (meaning that it is the pp-type of some element in an indecomposable pure-injective).  We check Ziegler's criterion (\cite[4.4]{Zie}, see \cite[4.3.49]{PreNBK}).  For that, we take any pp formulas $D\phi_1, D\phi_2 \notin p$ that is, as shown above, with $\phi_1, \phi_2 \notin  {\cal I}$.  By maximality of ${\cal I}$ there are $\psi_1, \psi_2 \in {\cal I}$ such that $a\in (\psi_i+\phi_i)(M)$ ($i=1,2$) so, setting $\psi =\psi_1+\psi_2$, we have $a \in (\psi+\phi_1)(M)\,\cap \,(\psi+\phi_2)(M)$.  Since $f(a)\neq 0$, \ref{dualann} gives $f\notin D((\psi+\phi_1)\wedge (\psi+\phi_2))(M^\ast)$.  That is, $f\notin ((D\psi\wedge D\phi_1)+(D\psi\wedge D\phi_2))(M^\ast)$.  That is, we have $D\psi \in p$ such that $(D\psi\wedge D\phi_1)+(D\psi\wedge D\phi_2)\notin p$, and that is Ziegler's criterion, so our claim is established.

Therefore any element with pp-type $p$ in a pure-injective module is contained in an indecomposable summand of that module (see \cite[\S4.3.5, esp.\,4.3.46]{PreNBK}); applied to $f\in M^\ast$, we have the theorem.
\end{proof}

Here is our example.

\begin{example}\label{notdual}\marginpar{notdual} Let $R$ be the free associative algebra $K\langle X,Y\rangle$ over a field $K$.  Then $R$ is a domain with no uniform one-sided ideal, so its injective hull $E$ has no indecomposable direct summand.  Nor does any product of copies of $E$ have a direct summand, since any nonzero submodule of a product of copies of $E$ must embed a copy of $R$.  It follows by \ref{dualindec} that $(0,{\rm Prod}(E))$ is an almost dual pair, with the first class in no sense determining the second.
\end{example}

This shows that not every almost dual pair has the form $({\cal S}, {\cal P})$ with ${\rm Prod}^+({\cal S}^\ast) \subseteq {\cal P} \subseteq {\rm Prod}({\cal S}^\ast)$, in particular not all classes ${\cal P}$ as in \ref{tpdef} arise from dual pairs, equivalently not all torsion theories on the functor category arise this way.  In particular, from \ref{prodcog} below we see that the torsion theories which arise from dual pairs of the form $({\cal S}, {\rm Prod}({\cal S}^\ast))$ are cogenerated by indecomposable injectives.

If ${\cal D}$ is a definable subcategory of ${\rm Mod}\mbox{-}R$ and $N\in {\rm Pinj}({\cal D})$ is indecomposable then $N$ is {\bf neg-isolated} in ${\cal D}$ if it is the hull of a pp-type which is neg-isolated modulo ${\cal D}$, equivalently if, whenever $N$ is a direct summand of a product $\prod_\lambda N_\lambda$ in ${\rm Pinj}({\cal D})$ already $N$ is a direct summand of some $N_\lambda$ (see \cite[\S 5.3.5, esp.\,5.3.48]{PreNBK}); another equivalent is that $(N\otimes -)$ is the injective hull of a functor in $(R\mbox{-}{\rm mod}, {\bf Ab})$ which becomes a simple object in the (finite type) localisation of that functor category at the torsion theory cogenerated by the $(N'\otimes -)$ with $N'\in {\rm Pinj}({\cal D})$ (\cite[12.5.6]{PreNBK}).

\begin{prop}\label{prodcog}\marginpar{prodcog} Let $M$ be any nonzero module and let $M^\ast$ be a dual of $M$.  Then ${\rm Prod}(M^\ast)$ is cogenerated by indecomposable pure-injectives, indeed ${\rm Prod}(M^\ast)= {\rm Prod}({\cal N})$ where ${\cal N}$ is the set of direct summands of $M^\ast$ which are neg-isolated in the definable category generated by $M^\ast$.
\end{prop}
\begin{proof} We get this by combining the arguments of \ref{anydual} and \ref{dualindec}.  Take $f\in M^\ast$, $f\neq 0$ and let $p$ be the pp-type of $f$ in $M^\ast$.  As before, for each $\psi \in p^-$ we have $D\psi(M)\nleq \sum_{\phi\in p}D\phi(M)$.  Choose $a_\psi \in D\psi(M)\setminus \sum_{\phi\in p}D\phi(M)$.  As in the proof of \ref{dualindec} choose a lattice ideal ${\cal I}$ in ${\rm pp}(M)$ which contains all the $D\phi(M)$ with $\phi\in p$ and such that $a\notin \sum_{D\phi\in {\cal I}}D\phi(M)$.  Then there is $f_\psi\in M^\ast$ such that $f_\psi(\sum_{D\phi\in {\cal I}}D\phi(M))=0$ - hence such that $f_\psi\in \phi(M^\ast)$ for each $D\phi\in {\cal I}$ - and such that $f(a)\neq 0$ so, by \ref{dualann}, $f\notin D\psi(M^\ast)$.  As in the proof of \ref{dualindec} the pp-type of $f_\psi$ is irreducible, indeed neg-isolated in the theory of $M^\ast$ by $\psi$ (this means that ${\cal I}$ is determined uniquely as a lattice ideal by the fact that it contains ${\cal I}$ and does not contain $\psi$).  This means that if we choose any minimal direct summand $N_\psi$ of $M^\ast$ which contains $f_\psi$ then this is indecomposable and neg-isolated in the definable subcategory generated by $M^\ast$.  Then, just as in the proof of \ref{anydual}, we deduce that $M^\ast $ embeds into a product of such neg-isolated pure-injectives, which is enough.
\end{proof}

\section{Definable subcategories}

There is a natural bijection, elementary duality (\cite[6.6]{HerzDual}, see \cite[\S 3.4.2]{PreNBK}) between the definable subcategories of $R\mbox{-}{\rm Mod}$ and those of ${\rm Mod}\mbox{-}R$ which can be defined in various ways, the most natural here being to take a definable subcategory ${\mathcal D}$ of $R\mbox{-}{\rm Mod}$ to ${\mathcal D}^{\rm d} =({\mathcal D}^\ast)^+$, which is a definable subcategory of ${\rm Mod}\mbox{-}R$ (see \cite[1.3.15]{PreNBK})) and, see \cite[3.4.21]{PreNBK}, $({\cal D}^{\rm d})^{\rm d} ={\cal D}$.  The next observation, therefore is immediate.

\begin{remark} \label{dualpairs}\marginpar{dualpairs} If ${\cal D}$ is a definable category of right $R$-modules and ${\cal D}^{\rm d}$ denotes the elementary dual definable category of left $R$-modules then both $({\cal D}^{\rm d}, {\cal D})$ and $({\cal D}, {\cal D}^{\rm d})$ are dual pairs.  We will refer to any such pair as a dual pair of definable categories.
\end{remark}

Here is an example of a dual pair which is not definable.

\begin{example}\label{notdualdual}\marginpar{notdualdual} Consider the almost dual pair $({\rm Add}({\mathbb Z}_{p^\infty}), {\rm Prod}(\overline{{\mathbb Z}_{(p)}}))$ which is ``cogenerated" by the p-adic integers $\overline{{\mathbb Z}_{(p)}}$ regarded as a ${\mathbb Z}$-module. This is a dual pair, since $({\rm Add}({\mathbb Z}_{p^\infty}))^\ast = {\rm Prod}(\overline{{\mathbb Z}_{(p)}})$, and can equally be regarded as being ``generated" by the Pr\"{u}fer group ${\mathbb Z}_{p^\infty}$.

The dual of $\overline{{\mathbb Z}_{(p)}}$ is ${\mathbb Z}_{p^\infty}\oplus {\mathbb Q}^{(2^{\aleph_0})}$, which is not in ${\rm Add}({\mathbb Z}_{p^\infty})$ so this is not a dual pair of definable subcategories (cf.~\ref{def}).  Indeed the dual pair consisting of definable subcategories which minimally contains this pair is $({\rm Add}({\mathbb Z}_{p^\infty}\oplus {\mathbb Q}), {\rm Prod}^+(\overline{{\mathbb Z}_{(p)}} \oplus {\mathbb Q}))$.
\end{example}

\begin{prop}\label{Sdef}\marginpar{Sdef} (\cite[2.2]{LenzTrans}, \cite[4.2]{CBlfp}, see \cite[4.1]{Holm}) Suppose that $({\mathcal S},{\mathcal P})$ is an almost dual pair over $R$.  Then the following are equivalent:

\noindent (i)  ${\mathcal S}$ is definable;

\noindent (ii) ${\cal S}$ is closed under products;

\noindent (iii) ${\cal S}$ is preenveloping in $R\mbox{-}{\rm Mod}$.
\end{prop}

Example \ref{notdual} shows that if $({\cal S}, {\cal P})$ is an almost dual pair and ${\cal S}$ is definable then it need not be that ${\cal P}$ is definable (that is, closed under direct limits, equivalently under pure epimorphisms).  On the other hand definability of ${\cal P}$ (which implies ${\cal P}^+={\cal P}$) does imply definability of ${\cal S}$.

\begin{theorem}\label{def}\marginpar{def} Suppose that $({\mathcal S},{\mathcal P})$ is an almost dual pair over $R$.  Then the following are equivalent:

\noindent (i)  ${\mathcal P}^+$ is definable;

\noindent (ii)  ${\mathcal P}^\ast \subseteq {\mathcal S}$;

\noindent (iii)  $({\mathcal P}^+)^\ast  \subseteq {\mathcal S}$

\noindent (iv)  ${\mathcal S}$ is definable and ${\rm Pinj}({\mathcal P}) \subseteq {\rm Prod}({\mathcal S}^\ast)$;

\noindent (v)  ${\mathcal S}$ is definable and every $A \in {\mathcal P}^+$ is pure in the dual of some module from ${\cal S}$;

\noindent (vi) ${\cal S}$ is definable and $({\cal S}, {\cal P})$ is a dual pair;

\noindent (vii) $({\cal S}, {\cal P}^+)$ (and hence also $({\cal P}^+, {\cal S})$) is a dual pair of definable subcategories.
\end{theorem}
\begin{proof}  (i)$\Rightarrow$(iii)  Let $A\in {\mathcal P}^+$; since ${\mathcal P}^+$ is definable, $A{^\ast} ^\ast \in {\mathcal P}^+$ so, being both pure-injective and pure in some member of ${\mathcal P}$, $A{^\ast} ^\ast$ is in ${\mathcal P}$.  Therefore, by the definition of almost dual pair, $A^\ast \in {\mathcal S}$.

Clearly (iii)$\Rightarrow$(ii).

(ii)$\Rightarrow$(iv) The second condition follows from the fact that any pure-injective is a direct summand of its double-dual.  To show that ${\cal S}$ is definable it will be enough, by \ref{Sdef}, to show that ${\cal S}$ is closed under direct products, so take $A_i\in {\cal S}$, $i\in I$.  Then for each $i$, $A_i^\ast \in {\cal P}$ so $\prod_iA_i^\ast \in {\cal P}$.  Since the canonical embedding $\bigoplus_iA_i^\ast \rightarrow \prod_iA_i^\ast$ is pure, the dual map gives $(\bigoplus_iA_i^\ast)^\ast$ as a direct summand of $(\prod_iA_i^\ast)^\ast$ which, by assumption, is in ${\cal S}$.  Also by assumption each ${A_i^\ast} ^\ast$  is in ${\cal S}$.  So $\prod_iA_i$, which is pure in $\prod_i{A_i^\ast} ^\ast \simeq (\bigoplus_iA_i^\ast)^\ast$, is, by \ref{Sclosed}, in ${\cal S}$, as required.

(iv)$\Rightarrow$(i) Let $M\in {\cal P}^+$, say $M$ is pure in $N\in {\rm Pinj}({\cal P})$.  By assumption there is $B\in {\cal S}$ with $N$ a direct summand of $B^\ast$, so $M$ purely embeds in $B^\ast$ and, dualising, $M^\ast$ is a direct summand of ${B^\ast}^\ast$ which, since ${\cal S}$ is definable, is in ${\cal S}$.  Hence $M^\ast \in {\cal S}$ and we have $({\cal P}^+)^\ast \subseteq {\cal S}$.

Suppose that we have a directed system $(A_i)_{i\in I}$ in ${\mathcal P}^+$ with direct limit $A$.  As in \ref{Sclosed}, $A$ is a pure epimorphic image of the direct sum of the $A_i$ and, dualising, we obtain a split embedding $A^\ast \rightarrow \big( \bigoplus_i A_i\big)^\ast = \prod_i A_i^\ast$.  We have just seen that each $A_i^\ast$ is in ${\mathcal S}$, hence $A^\ast \in {\mathcal S}$ and therefore $A{^\ast} ^\ast \in {\cal P}$.  Therefore $A\in {\mathcal P}^+$, as required.

(iv)$\Leftrightarrow$(v) This is immediate.

(iv)($\Rightarrow$)(vi) is immediate from the definitions, as is (i)+(iv)($\Rightarrow$)(vii).  Both (vi)($\Rightarrow$)(iv) and (vii)($\Rightarrow$)(i) are immediate.
\end{proof}

Example  \ref{notdualdual} shows that definability of ${\cal S}$ cannot be dropped from condition (iv).

\begin{cor} \label{defdualdual}\marginpar{defdualdual} If $({\cal S}, {\cal P})$ is an almost dual pair of $R$-modules and ${\cal P}$ is a definable subcategory of ${\rm Mod}\mbox{-}R$ then $({\cal P}, {\cal S})$ also is a(n almost) dual pair of $R$-modules, and ${\cal S}=({\cal P}^\ast)^+$.
\end{cor}

\begin{cor} \label{defpr}\marginpar{defpr} If ${\cal S}$ is a definable subcategory of $R\mbox{-}{\rm Mod}$ then both $({\cal S}, {\rm Prod}^+({\cal S}^\ast))$ and $({\rm Prod}^+({\cal S}^\ast), {\cal S})$ are dual pairs of definable categories.
\end{cor}

\section{Almost dual pairs generated by finitely presented modules}\label{secfp}\marginpar{secfp}

Holm showed that if ${\mathcal B}$ is an additive subcategory of $R\mbox{-}{\rm mod}$ then $(\varinjlim {\mathcal B}, {\rm Prod}({\mathcal B}^\ast))$ is a(n almost) dual pair.  We will give a somewhat modified proof of this here.

\begin{theorem}\label{bijft}\marginpar{bijft} (\cite[1.4]{Holm}) Let ${\mathcal B}$ be an additive subcategory of $R\mbox{-}{\rm mod}$.  Then $M\in \varinjlim {\mathcal B}$ iff $M^\ast \in {\rm Prod}({\mathcal B}^\ast)$ and hence $(\varinjlim {\mathcal B}, {\rm Prod}^+({\mathcal B}^\ast))$ is a dual pair, as is $(\varinjlim {\mathcal B}, {\rm Prod}({\mathcal B}^\ast))$.
\end{theorem}
\begin{proof} If $M\in \varinjlim {\mathcal B}$ then, as in the proof of \ref{Sclosed}, there is a pure epimorphism $\bigoplus_iB_i\rightarrow M$ with the $B_i$ in ${\cal B}$, and hence a split embedding $M^\ast \rightarrow \prod_iB_i^\ast$.  This proves the direction ($\Rightarrow$).

For the other direction, we follow \cite[4.2.18, 4.2.19]{MehdiThes}.  Suppose that $M\in R\mbox{-}{\rm Mod}$ is such that $M^\ast \in {\rm Prod}({\cal B}^\ast)$, say $i:M^\ast \rightarrow \prod_iB_i^\ast$ with $B_i\in {\cal B}$ is a split embedding.  We may assume that the duality $^\ast$ is with respect to the injective cogenerator $E$ of $S$-modules, where $S$ maps to the centre of $R$; thus all hom groups between $R$-modules are $S$-modules.

By \cite[3.2]{ElB} ${\rm Add}({\cal B})$ is precovering so choose a precover $\bigoplus_jA_j \xrightarrow{\alpha} M$ with the $A_j\in {\rm Add}({\cal B})$.  We will show that $\alpha$ is surjective.  Since $\alpha$ is a precover, for each $A\in {\cal B}$ the induced map $(A,\bigoplus_jA_j) \rightarrow (A,M)$ is surjective.  Therefore the induced map ${\rm Hom}_S((A,M),E) \rightarrow {\rm Hom}_S((A,\bigoplus_jA_j),E)$ is injective.  Since $A$ is finitely presented we have (e.g.~\cite[25.5(ii)]{WisBk}) the natural isomorphism ${\rm Hom}_S({\rm Hom}_R(A,-),E) \simeq {\rm Hom}_S(-,E)\otimes_RA =(-)^\ast \otimes_RA$, so we have that the map $M^\ast\otimes_RA \rightarrow (\bigoplus_jA_j)^\ast\otimes_RA$ induced by $\alpha$ is injective.  These are $S$-modules so, by injectivity of $E$ we have that the induced map ${\rm Hom}_S((\bigoplus_jA_j)^\ast\otimes_RA,E) \rightarrow {\rm Hom}_S(M^\ast\otimes_RA,E)$ is surjective and hence (by the ${\rm Hom}$/$\otimes$ adjunction) that ${\rm Hom}_R((\bigoplus_jA_j)^\ast,A^\ast) \rightarrow {\rm Hom}_R(M^\ast,A^\ast)$ is surjective.

This is so for each $A\in {\cal B}$, in particular for each $B_i$, so we deduce that the map $\prod_i ((\bigoplus_jA_j)^\ast,B_i^\ast) \rightarrow \prod_i (M^\ast,B_i^\ast)$ is surjective, hence that the induced map $ ((\bigoplus_jA_j)^\ast,\prod_iB_i^\ast) \rightarrow (M^\ast,\prod_iB_i^\ast)$ is surjective.  But the latter is the map induced by $\alpha^\ast: M^\ast \rightarrow (\bigoplus_jA_j)^\ast$ and so we deduce that $i\in (M^\ast,\prod_iB_i^\ast)$ factors through $\alpha^\ast$ and hence, since $i$ is monic, $\alpha^\ast$ is monic, indeed it is a pure, hence split, embedding.

Finally we use that $E$ is a cogenerator:  if $\alpha$ were not surjective then its cokernel would have a non-zero map to $E$ and this, by composition with $M\rightarrow {\rm coker}(\alpha)$, would give a non-zero element of $M^\ast$ sent to $0$ by $\alpha^\ast$, which would contradict what we have just shown.

We deduce that $\alpha$ is indeed an epimorphism and so the sequence $0\rightarrow {\rm ker}(\alpha) \rightarrow \bigoplus_jA_j \xrightarrow{\alpha} M \rightarrow 0$ is exact.  We have also seen that $\alpha^\ast$ is split, that is, the dual sequence $0\rightarrow M^\ast \rightarrow (\bigoplus_jA_j)^\ast \rightarrow ({\rm ker}(\alpha))^\ast \rightarrow 0$ is split.  Therefore the original sequence is pure, in particular $\alpha$ is a pure epimorphism and so (see the proof of \ref{Sclosed}), $M\in \varinjlim {\cal B}$, as required.
\end{proof}

\begin{examples}\label{exacc}\marginpar{exacc}  If $ {\mathcal B} ={\rm add}({\mathcal B})$ is a subcategory of $R\mbox{-}{\rm mod} $ then (see \cite{AdRo} or \cite{CBlfp}) $\varinjlim {\mathcal B}$ is a finitely accessible category and hence is a definable category in the sense of \cite{PreNBK}, \cite{PreDefAddCat} but it might not be a definable {\em subcategory} of $ R\mbox{-}{\rm Mod}$: it will be closed in $ R\mbox{-}{\rm Mod}$ under direct limits and pure submodules but it might not be closed under direct products. Indeed, the product of any collection of modules in $ \varinjlim {\mathcal B} $ will have a maximal, possibly proper, submodule which lies in $ \varinjlim {\mathcal B} $ and that will give the product within the category $ \varinjlim {\mathcal B}$. For an illustrative example, take $ {\mathcal B} $ to be the category of finite abelian groups, so $ \varinjlim {\mathcal B} $ is the category of torsion abelian groups; in this case we have the almost dual pair $(\varinjlim {\mathcal B}, (\varinjlim {\mathcal B})^\ast)$ whose second component is the category of profinite abelian groups.

In contrast, if we take $ {\mathcal B} $ to be the category of preprojective left modules over a tame hereditary algebra, then $ \varinjlim {\mathcal B} $ is the category of left modules which are torsionfree in the sense of \cite{RinInf} (see \cite[p.~743]{LenzTrans}) and ${\mathcal P}$ is the class of right modules which are divisible in the sense of that paper; each is a definable subcategory of the respective category of all modules (e.g.~by \cite[3.2]{Zad} and \ref{Sdef} or the above references).

If $R$ is a finite-dimensional algebra then the classes of the form ${\rm Prod}(({\mathcal B}\cup \{ _RR\})^\ast)$ for ${\cal B}$ an additive subcategory of $R\mbox{-}{\rm mod}$ are the classes of relative pure-injectives for purities determined by sets of finitely presented modules (see \cite[4.5]{Meh}).
\end{examples}

Holm says that $R$ is ``${\cal B}$-coherent'' if the equivalent conditions (i), (ii), (iii) of \ref{Sdef} are satisfied.  The results \ref{def} and \ref{bijft} here add to this and to \cite[5.6, 5.7]{Holm}, at the same time removing the additional, but as it turns out unnecessary, condition $R\in {\cal B}$ from the latter two results.

\vspace{4pt}

Holm \cite[1.3]{Holm} also considers the stronger condition (``${\cal B}$-noetherian") that ${\rm Prod}({\mathcal B}^\ast ) $ be a definable subcategory of $ {\rm Mod}\mbox{-}R $.  This is equivalent to the condition that every member of $ ({\rm Prod}({\mathcal B}^\ast )) $ be $ \Sigma $-pure-injective, where a module $M$ is said to be $\Sigma${\bf -pure-injective} if $M^{(I)}$ is pure-injective for any (and then it follows, for every) infinite set $I$.  We add, to the various characterisations \cite[1.3]{Holm} of this case, the following.

\begin{theorem}\label{noeth}\marginpar{noeth} Let $ {\mathcal B} $ be an additive subcategory of $ R\mbox{-}{\rm mod}. $ Set $ M=\bigoplus \{B: B\in {\mathcal B}' \} $ where $ {\mathcal B}'  $ is a skeletally small version of $ {\mathcal B}. $ Then the following conditions are equivalent:

\noindent (i) $ {\rm Prod}({\mathcal B}^\ast ) $ is definable;

\noindent (ii) $ M^\ast =\prod\{B^\ast : B\in {\mathcal B}' \} $ is $ \Sigma$-pure-injective;

\noindent (iii) $ M $ is noetherian over its endomorphism ring.

\noindent In this case ${\rm Prod}({\mathcal B}^\ast ) = {\rm Prod}^+({\mathcal B}^\ast )$ is definable and is the elementary dual of the definable subcategory $ \varinjlim {\mathcal B}$.
\end{theorem}
\begin{proof} (i)$\Rightarrow$(ii) is immediate from the fact that a definable subcategory is closed under direct sums.

(ii)$\Rightarrow$(i)  It is well-known (e.g.~\cite[9.34]{PreBk}) and easy to prove (use the result quoted in the proof of \ref{Sclosed}) that if a module $M$ is $\Sigma$-pure-injective then ${\rm Prod}(M)$ is a definable subcategory of ${\rm Mod}\mbox{-}R$.

(i)$\Leftrightarrow$(iii)  It is shown in \cite[Observation 8, p.~705]{Z-HZPSS} that if a module is (a direct summand of) a direct sum of finitely presented modules then it is noetherian over its endomorphism ring iff it has the ascending chain condition on pp-definable subgroups.  That, by \cite[Lemma 2, Prop.~2]{Z-HZPSS} (also see \cite[1.3.15]{PreNBK}) is true of a module iff its Hom-dual has the descending chain condition on pp-definable subgroups and that, in turn, is equivalent to this Hom-dual being $\Sigma$-pure-injective (see \cite[4.4.5]{PreNBK} for a proof and sources for this last result).

The last statement follows by \ref{bijft}.
\end{proof}

An example here is $(R\mbox{-}{\rm Flat}, {\rm Inj}\mbox{-}R)$ where $R$ is a right noetherian ring.  This example illustrates that, in this situation, ${\cal S}$ need not equal ${\rm Pinj}({\cal S})$.  In fact, the case where both a definable category and its dual consist only of ($\Sigma$-)pure-injectives is exactly that where all the modules in these classes have finite endolength (e.g.~see \cite[4.2.25 and 1.3.15]{PreNBK}).

\begin{prop}\label{fel}\marginpar{fel}  If $({\cal S}, {\cal P})$ is an almost dual pair then the following conditions are equivalent:

\noindent (i) both ${\cal P}={\rm Prod}({\cal S}^\ast)$ and ${\cal S}={\rm Prod}({\cal S}^\ast)$;

\noindent (ii) $M^\ast$ as defined in \ref{noeth} is of finite length over its endomorphism ring.
\end{prop}

\begin{example}\label{}\marginpar{} Suppose that $ R $ is an artin algebra and that $ {\mathcal B} $ is an additive subcategory of $ R\mbox{-}{\rm mod}. $ If $ {\mathcal B} $ is closed under submodules then $ \varinjlim {\mathcal B} $ is definable \cite[2.2]{KraPre}.  By \ref{bijft} and \ref{defpr} and the fact that $^\ast$ is a duality between $R\mbox{-}{\rm mod}$ and ${\rm mod}\mbox{-}R$ we deduce that if ${\cal B}'$ is any additive category of ${\rm mod}\mbox{-}R$ closed under quotients then ${\rm Prod}^+({\cal B}')$ is a definable subcategory of ${\rm Mod}\mbox{-}R$.

The pair $(\varinjlim {\mathcal B}, {\rm Prod}^+({\cal B}^\ast))$ will satisfy the conditions of \ref{noeth} iff ${\cal B}$ contains only finitely many indecomposable modules up to isomorphism (by \cite[4.4.31]{PreNBK}).

If, as in Example \ref{exacc}, we take $R$ to be a tame hereditary algebra and we take ${\cal P}_0$ to be the set of  indecomposable preprojective left $R$-modules and ${\cal I}_0$ to be the set of indecomposable preinjective right $R$-modules then we have the dual pair of definable categories $(\varinjlim {\cal P}_0, {\rm Prod}^+({\cal I}_0))$ (as well as $({\rm Prod}^+({\cal I}_0), \varinjlim {\cal P}_0)$) but only the second class, ${\rm Prod}^+({\cal I}_0)$, satisfies the condition of \ref{noeth} (unless $R$ is of finite representation type).
\end{example}

\end{document}